\documentclass[VANCOUVER,STIX1COL]{WileyNJD-v2}
\usepackage{hyperref}
\usepackage{pxfonts}
\usepackage{subfigure}
\usepackage{amssymb}
\newtheorem{mypro}{Proposition}[section]
\begin{document}
\articletype{Article Type}%
\raggedbottom
\title{Techniques for Accelerating the Convergence of Restarted GMRES Based on the Projection
\protect\thanks{Supported by National Natural Science Foundation of China (11101071, 11271001) and Fundamental Research Funds for the Central Universities (ZYGX2016J131, ZYGX2016J138).}}

\author[1]{Hou-biao Li*}

\author[1]{Peng-hui He}

\author[2]{Shao-Liang Zhang}

\authormark{Hou-biao Li \textsc{et al}}

\address[1]{\orgdiv{School of Mathematical Sciences}, \orgname{University of Electronic Science and Technology of China}, \orgaddress{\state{Sichuan}, \country{China}}}

\address[2]{\orgdiv{Department of Applied Physics}, \orgname{Nagoya University}, \orgaddress{\state{Nagoya}, \country{Japan}}}


\corres{*Hou-biao Li, No.2006, Xiyuan Ave, West Hi-Tech Zone, 611731. \email{lihoubiao0189@163.com}}

\presentaddress{No.2006, Xiyuan Ave, West Hi-Tech Zone, 611731,China}

\abstract[Summary]{In this paper, we study the restarted Krylov subspace method, which is typically represented by the GMRES(m) method. Our work mainly focused on the amount of change in the iterative solution of GMRES(m) at each restart. We propose an extension of the GMRES(m) method based on the idea of projection. The algorithm is named as LGMRES. In addition, LLBGMRE method is also obtained by adding backtracking restart technology to LGMRES. Theoretical analysis and numerical experiments show that LGMRES and LLBGMRES have better convergence than traditional restart GMRES(m) method.}



\maketitle


\section{Introduction}

In recent year, there has been extensive research on Krylov subspace method for solving large and sparse linear systems of the form
\begin{equation}\label{1.1}
  Ax=b,A \in \mathbb{R}^{n\times n}, x,b \in \mathbb{R}^n.
\end{equation}
Assuming that the coefficient matrix $A$ is non-Hermitian and non-singular.
The above linear equations are often generated by the discretization of partial differential equations. In computational science and engineering, the coefficient matrices of linear algebraic equations are often constructed into a large sparse matrix.

The iterative method is the most common method to solve large, sparse and non-Hermitian linear systems (\ref{1.1}). Recently, the Krylov subspace methods \cite{2} such as FOM(m)\cite{3} method and GMRES(m) method \cite{4} are recognized as standard algorithms for this kind of linear systems, for details see \cite{1}. For GMRES(m) methods, when the matrix $A$ is symmetric, convergence depends only on its eigenvalues, when $A$ is nonsymmetric but diagonalizable, convergence depends on the eigenvalues and the eigenvectors \cite{5}. In addition, the (two-sided) Lanczos-based Krylov subspace Methods also have some difficulties in terms of the pseudo convergence due to accumulations of the round-off errors. In order to remedy these difficulties, the restart is often applied to the Krylov subspace methods. In this paper, we investigate the restarted Krylov subspace methods, as typified by the GMRES(m) method. But the restart generally slows their convergence, some ways to accelerate convergence are considered by the update of the initial, for more details see \cite{6,7}. For example, the initial of Look-back-Type restart method is obtained by a special method in \cite{1}. The preconditioning technology is used to speed up the GMRES method \cite{8,9,10,11,12}. The multipreconditioned GMRES (MPGMRES)\cite{13} is also proposed by using two or more preconditioners simultaneously. Meanwhile, a selective version of the MPGMRES is presented to overcome the rapid increase of the storage requirements and make it practical.

Based on GMRES method and Look-back-type method, this paper accelerates the above two algorithms again. LGMRES method and LLBGMRES method are proposed respectively. The principle of GMRES algorithm is to project residual error into a space, so that the length of residuals is decreasing. Every projection will get a projection vector. Accelerate the GMRES method algorithm by using the relationship between adjacent projection vectors. With this idea, we accelerate the GMRES(m) method and Look-Back-type method \cite{1}. The performance of LGMRES method and LLBGMRES method is evaluated by a comparison analysis based on the theoretical analysis and some numerical experiments.

This paper is organized as follows. In Section 2, we briefly describe a general form of the restarted Krylov subspace methods and GMRES(m) method. The restarted GMRES(m) method and a Look-Back-type restart are introduced in Section 3. For GMRES(m) and Look-back GMRES(m) methods take an acceleration strategy and put forward their deformable body LGMRES and LLBGMRES methods in Section 4. We analyze the convergence of the algorithm and prove that the modified algorithm is convergent in the Section 4. In Section 5, the performance of the LGMRES and LLBGMRES methods is evaluated by some numerical experiments. Finally, our conclusions are summarized in Section 6.

\section{The GMRES(m)}

Let $x_{0}$  denote an initial guess for the solution of system (\ref{1.1}), and $r_{0}=b-Ax_{0}$ the corresponding initial residual.
Choose two space of dimension $1\leq m \leq n$, $\mathbb{K}_{m}$ and $\mathbb{L}_{m}$, Petrov-Galerkin method solves $Ax=b$ by requiring:
$$x_{m}\in x_{0}+\mathbb{K}_{m},$$
$$b-Ax_{m}\perp \mathbb{L}_{m}.$$
A Krylov method takes
$$\mathbb{K}_{m}=\mathcal{K}_{m}(A,r_{0})=span\{r_{0},Ar_{0},...,A^{m-1}r_{0}\},$$
which is called as a Krylov space.
The central idea of Krylov method is projection. GMRES(m) method is a special case of Krylov method. It is oblique projection.
$$\mathbb{L}_{m}=A\mathbb{K}_{m}.$$

Let $V$ and $W$ be the subspaces $\mathbb{K}_{m}$ and $\mathbb{L}_{m}$ respectively, Solve
\begin{equation}\label{1.2}
W^{T}AVz=W^{T}r_{0},
\end{equation}

\begin{equation}\label{1.3}
x=x_{0}+Vz,
\end{equation}

\begin{equation}\label{1.4}
r=b-Ax=b-A(x+Vz)=r_{0}-AVz=r_{0}-Wz,
\end{equation}
and $$r_{0}-Wz\perp W.$$

This is equivalent to the projection of $r_{0}$ on the $\mathbb{L}_{m}$. $Wz$ is the projection of $r_{0}$ on $\mathbb{L}_{m}$. $r$ is perpendicular to $\mathbb{L}_{m}$. This ensures that the length of residuals in the whole iteration process is decreasing. According to the principle of projection, we are easy to get

$(1)$ If $r_{0}\perp \mathbb{L}_{m}$, then $r=r_{0}$, the reduction of residual model length is 0.

$(2)$ If $r_{0}\in \mathbb{L}_{m}$, then $r=0$, the corresponding $x$ is the exact solution.

$(3)$ If the angle between $r_{0}$ and$\mathbb{L}_{m}$ becomes smaller, the reduction of residuals modulus length is larger.

\begin{tabular}{l}
\hline
\multirow{1}{8cm}{Algorithm 1.  The GMRES(m) method \cite{18}}\\
  \hline
1: Chose the restart frequency $m$ and the initial guess $x_{0}^{(1)}$\\
2: Compute $r_{0}=b-Ax_{0}^{1}$ and set $\beta=\ \ \| r_{0}\|$\\
3: for $j=1,2,...,m$, do\\
4:     Compute $w_{j}=Av_{j}$\\
5:     for $i=1,2,...,j$\\
6: $h_{i,j}=(w_{j},v_{i})$\\
7: $w_{j}=w_{j}-h_{i,j}v_{i}$\\
8:     end for\\
9: $h_{j+1,j}=\ \ \| w_{j}\|$\\
10: $v_{j+1}=w_{j}/h_{j+1,j}$\\
11: end for\\
12: Define the $(m+1)\times m$ Hessenberg matrix $H_{m}=\{h_{i,j}\}_{1\leq i \leq m+1,1 \leq j \leq m}$\\
13: $z^{l}=V_{m}s_{m}$ ,where $s_{m}=\mathop{\arg\min}_{s\in \mathbb{R}^{m}} \ \ \| \beta e_{1}-H_{m}s\|_{2}$\\
14: $x^{(t)}=x_{0}^{(t)}+z^{(t)}$\\
14: Computer $r^{(t)}=b-Ax^{(t)}$. If convergence then Stop\\
15: Update $x_{0}^{(t+1)}=x^{(t)}$, go on to 2\\
  \hline
\end{tabular}
\\

To analyze the convergence property of GMRES(m), the Cayley-Hamilton theorem can be applied: for any $A\in \mathbb{P}^{n \times n}$, $f(\lambda)=|\lambda E-A|$ is its characteristic polynomial, then $f(A)=0$. So for any $A\in \mathbb{P}^{n \times n}$, we can always find a smallest polynomial that makes $f(A)=0$. Here we use $P_{d+1}(\lambda)$ to represent the smallest polynomials of $A$, i.e., $P_{d+1}(A)=0$, and
$$\forall P_{k} \in \mathbb{P}_{d}, P_{k}(A)\neq 0, k\leq d,$$
where $P_{d}$ is the set of $d$-degree polynomials. Define
$$P_{d+1}(A)=a_{0}E+a_{1}A+...+a_{d}A^{d}=0.$$

For convenience, we can change the constant coefficient of the minimum polynomial to 1, i.e., $a_{0}=1$. If $A$ is a nonsingular matrix, then
$$A^{-1}=a_{1}E+a_{2}A+..+a_{d}A^{d-1},$$
$$x^{*}=A^{-1}b=(a_{1}E+a_{2}A+...+a_{d}A^{d-1})b,$$
$$x^{*}=x_{0}+(a_{1}E+a_{2}A+...+a_{d}A^{d-1})r_{0}.$$

So $x^{*}\in x_{0}+\mathcal{K}_{d}(A,r_{0})$, and this ensures the feasibility of GMRES method and the residual is declining in the iterative process. This makes it possible for us to get the best approximation solution in Krylov subspace, see the more details in [5].

In addition, according to [5], ${r}={{r}_0}-{{Wz}} \in {{\cal K}_{{\rm{m}} + 1}}(A,{r_0})$, so there exists a polynomial ${\psi _m}(\lambda ) \in {P_m}$ such that ${{r}}={\psi_m}(A){r_0}$.
When the matrix $A$ is diagonalizable, there exists an nonsingular matrix $V$ and a diagonal matrix $\Lambda  = diag({\lambda _1},{\lambda _2},...,{\lambda _n})$ of the eigenvalues such that $A = V\Lambda {V^{ - 1}}$,
and we can further obtain
\[\left\| r \right\| \le \left\| {V{\psi _m}(A){V^{ - 1}}{r_0}} \right\| \le \kappa (V)\mathop {\min }\limits_{{\psi _m}(0) = 1} \mathop {\max }\limits_{1 \le i \le n} \left| {{\psi _m}({\lambda _i})} \right| \cdot \left\| {{r_0}} \right\|.\]
That is to say, the Euclidean norm of the residual is determined the eigenvalues and the eigenvectors of the matrix $A$. If matrix $A$ is symmetric, then matrix $V$ is an orthogonal matrix of the eigenvectors of the matric $A$, so $\kappa(V)=1$, then the convergence depends only on its eigenvalues.

\section{The Restarted GMRES(m) with a Look-Back-Type restart}
\subsection{The Restarted GMRES(m)}

Next, let us consider the restarted GMRES(m) method [6].
$x_{0}^{(t)}$ is the initial guess, $r_{0}=b-Ax_{0}^{(t)}$ is the sequence of the corresponding residual vectors.
$x^{(t)}=x_{0}^{(t)}+z^{(t)}$, $z^{(t)}\in \mathcal{K}_{m}(A,r_{0}^{(t)})$, where the vectors $z^{(t)}$ are designed by:
\begin{equation}\label{3.1}
  z^{(t)}=\mathop{\min}_{z \in \mathcal{K}_{m}(A,r_{0}^{(t)})} \ \ \| r_{0}^{(t)}-Az\|_{2}.
\end{equation}
So
 $$x^{(t)}\in x_{0}^{(1)}+\mathcal{K}_{m}(A,r_{0}^{(1)})+\mathcal{K}_{m}(A,r_{0}^{(2)})+...+\mathcal{K}_{m}(A,r_{0}^{(t)})\\=x_{0}^{(1)}+\mathcal{K}_{m\times t}(A,r_{0}^{(t)}).$$
In addition, Look-Back-type method [6] is proposed by adding the amended direction $y^{(t+1)}$, $$y^{(t+1)}\in\mathcal{K}_{m}(A,r_{0}^{(t)}),$$
$$x_{0}^{(t+1)}=x^{(t)}+y^{(t+1)},$$
to keep $x^{(t)}\in x_{0}^{(1)}+\mathcal{K}_{m\times t}(A,r_{0}^{(1)})$.

\subsection{A Look-Back-type restart}

By applying the sequence of $m^{(t)}$ dimensional Krylov subspaces ${{\cal K}_{{m^{({\rm{t}})}}}}(A,r_0^{(t)})$ for the linear systems (\ref{1.1}), the iterative solutions ${{{x}}^{(t)}}$ can be expressed as:
\begin{equation}\label{3.2}
  {x^{(t)}} = x_0^{(t)} + {z^{(t)}},{z^{(t)}} \in {{\cal K}_{{m^{({\rm{t}})}}}}(A,r_0^{(t)}),{\rm{   }}t = 1,2,....
\end{equation}
For more details, see [6]. And we have that
$${x^{(t)}} \in x_0^{(1)} + {{\cal K}_{{{\rm{M}}^{({\rm{t}})}}}}(A,r_0^{(1)}),{\rm{   }}{{\rm{M}}^{(t)}}{\rm{ = }}\sum\limits_{i = 1}^t {{m^{(i)}}} {\rm{   }}, t = 1,2,...$$
After amending the initial value $x_0^{(t + 1)} = {x^{(t)}} + {y^{(t + 1)}},t = 1,2,...$.
For any ${z^{(t)}} \in {{\cal K}_{{m^{({\rm{t}})}}}}(A,r_0^{(t)})$ (t=1,2,¡­),
$${x^{(t)}} \in x_0^{(1)} + {{\cal K}_{{{\rm{M}}^{({\rm{t}})}}}}(A,r_0^{(t)}),{\rm{   }}{{\rm{M}}^{(t)}}{\rm{ = }}\sum\limits_{i = 1}^t {{m^{(i)}}} {\rm{   }}, t = 1,2,...,$$
is satisfied if and only if
${{\rm{y}}^{(t)}} \in {{\cal K}_{{{\rm{M}}^{({\rm{t}})}}}}(A,r_0^{(1)}),{\rm{     }}t = 1,2,...,$
$${{\rm{r}}^{(t)}} = b - A{x^{(t)}},{x^{(t)}} \in x_0^{(1)} + {{\cal K}_{{{\rm{M}}^{({\rm{t}})}}}}(A,r_0^{(1)}),{\rm{t}}=1,2,...$$
So ${{\rm{r}}^{(t)}} \in {{\cal K}_{{{\rm{M}}^{({\rm{t}})}} + 1}}(A,r_0^{(1)})$. In addition, there exists a polynomial ${\rm{Q}}_{{{\rm{M}}^{(t)}}}^{(t)}(\lambda ) \in {\mathbb{P}_{{M^{(t)}}}}$   such that
$${r^{(l)}} = {\rm{Q}}_{{{\rm{M}}^{(t)}}}^{(t)}({\rm{A}}){\rm{r}}_0^{(1)}{\rm{,   Q}}_{{{\rm{M}}^{(t)}}}^{(t)}(0){\rm{=}}1.$$
${{\rm{y}}^{(t + 1)}} \in {{\cal K}_{{{\rm{M}}^{({\rm{t}})}}}}(A,r_0^{(1)})$ . There also exists a polynomial
${\rm{\widetilde{Q}}}_{{{\rm{M}}^{(t)}}}^{(t)}(\lambda ) \in {\mathbb{P}_{{M^{(t)}}}}$ such that
\[{\rm{\widetilde{Q}}}_{{{\rm{M}}^{(t)}}}^{(t)}(A){\rm{r}}_0^{(1)} = {\rm{Q}}_{{{\rm{M}}^{(t)}}}^{(t)}({\rm{A}}){\rm{r}}_0^{(1)} - A{y^{(t + 1)}}{\rm{,   \widetilde{Q}}}_{{{\rm{M}}^{(t)}}}^{(t)}{\rm{(0) = }}1.\]
So we can get
\[{y^{(t + 1)}} = {{\rm{A}}^{{\rm{ - }}1}}{\rm{(Q}}_{{{\rm{M}}^{(t)}}}^{(t)}({\rm{A}}){\rm{ -\widetilde{Q}}}_{{{\rm{M}}^{(t)}}}^{(t)}(A)){\rm{r}}_0^{(1)}.\]
The author represents ${\rm{\widetilde{Q}}}_{{{\rm{M}}^{(t)}}}^{(t)}(\lambda )$ as a linear combination of
${\rm{Q}}_{{{\rm{M}}^{(t)}}}^{(t)}(\lambda )$ and at most $M^{(t)}-$degree polynomial ${\rm{R}}_{{{\rm{M}}^{(t)}}}^{(t)}(\lambda ) \in {\mathbb{P}_{{M^{(t)}}}}{\rm{,R}}_{{{\rm{M}}^{(t)}}}^{(t)}(0) = 1$,i.e.,
\[{\rm{\widetilde{Q}}}_{{{\rm{M}}^{(t)}}}^{(t)}(\lambda ){\rm{ = }}{\tau ^{(t)}}{\rm{Q}}_{{{\rm{M}}^{(t)}}}^{(t)}{\rm{(}}\lambda {\rm{) + (1 - }}{\tau ^{(t)}}{\rm{)R}}_{{{\rm{M}}^{(t)}}}^{(t)}{\rm{(}}\lambda {\rm{), }}{\tau ^{(t)}} \in \mathbb{C}.\]
They defined the polynomial ${\rm{R}}_{{{\rm{M}}^{(t)}}}^{(t)}(\lambda )$ by using the $d$-th previous polynomial before
${\rm{Q}}_{{{\rm{M}}^{(t)}}}^{(t)}(\lambda )$ and ${\rm{\widetilde{Q}}}_{{{\rm{M}}^{(t)}}}^{(t)}(\lambda )$ as follows:
\[{\rm{R}}_{{{\rm{M}}^{(t)}}}^{(t)}(\lambda ):{\rm{ = }}\left\{ {\begin{array}{*{20}{c}}
{{\rm{Q}}_{{{\rm{M}}^{{\rm{(}}{{\rm{t}}_d})}}}^{({t_{\rm{d}}})}(\lambda ){\rm{    (}}d:even)}\\
{{\rm{\widetilde{Q}}}_{{{\rm{M}}^{{\rm{9}}{{\rm{t}}_d})}}}^{({t_{\rm{d}}})}(\lambda ){\rm{    (}}d:even)}
\end{array}} \right. {\rm{,  }}{{\rm{t}}_d}: = \left\{ {_{t - \frac{{d + 1}}{2}{\rm{   (d:odd)}}}^{t - \frac{d}{2}{\rm{(d:even)}}}} \right.{\rm{,d}} \in \mathbb{N}.\]
So
\[\begin{array}{l}
{y^{(t + 1)}} = (1{\rm{ - }}{\tau ^{(t)}}{\rm{)}}{{\rm{A}}^{{\rm{ - }}1}}{\rm{(Q}}_{{{\rm{M}}^{(t)}}}^{(t)}({{A}}){\rm{ - R}}_{{{\rm{M}}^{(t)}}}^{(t)}(A)){\rm{r}}_0^{(1)}\\
{\rm{=(}}1{\rm{ - }}{\tau ^{(t)}}{\rm{)}}{{{A}}^{{\rm{ - }}1}} \bullet \left\{ {\begin{array}{*{20}{c}}
{{{\rm{r}}^{(t)}} - {r^{({t_d})}}{\rm{    }}(d:even)}\\
{{{\rm{r}}^{(t)}} - r_0^{({t_d})}{\rm{    }}(d:{\rm{odd}})}
\end{array}} \right.
\end{array}\]
Let $\Delta {x^{({\rm{t}})}}$ and  ${u^{(t)}}$ be
\[\Delta {x^{({\rm{t}})}}: = \left\{ {\begin{array}{*{20}{c}}
{{{\rm{r}}^{(t)}} - {r^{({t_d})}}{\rm{    }}(d:even)}\\
{{{\rm{r}}^{(t)}} - r_0^{({t_d})}{\rm{    }}(d:{\rm{odd}})}
\end{array}} \right.,{u^{(t)}}{\rm{ = }}{\tau ^{(t)}}{\rm{ - }}1.\]
So we can get ${{\rm{y}}^{(t + 1)}} = {u^{(t)}}\Delta {x^{({\rm{t}})}}$. This is an unfixed update. The ${{\rm{y}}^{(t + 1)}} = {u^{(t)}}\Delta {x^{({\rm{t}})}}$ is set by Look-Back-type restart technology. The algorithm of the restarted Krylov subspace method with the Look-Back-type restart is shown in Algorithm 2.

\begin{table}[!http]
\begin{tabular}{l}
\hline
Algorithm 2. A restarted Krylov method with a Look-Back-type [6]\\
\hline
1: Chose the parameterd $\geq 2$ and the initial guess $x_{0}^{(1)}$\\
2: For $t=1,2,...$, until convergence Do:\\
3: Set the restart frequency $m^{(t)}$ and the $KS^{(t)}$ method\\
4: Solve (approxima) $Ax=b$ by $m^{(t)}$ iterations of the $KS^{t}$ method with the initial guess $x_{0}^{(t)}$, and \\ get the approximate solution $x^{t}$\\
5: Computer the vector $y^{(t+1)}$ as follows:\\
If $t=1$ then $y^{(t+1)}=0$\\
If $t\geq 2$  then\\
  If $(t=d=2)$ or (d:even, $t\leq \frac{d}{2}$ ) or (d:odd, $t\leq \frac{d-1}{2}$) then\\
$\Delta x^{(t)}=x^{(t)}-x_{0}^{(t_{d})}$\\
   Else\\
$ \Delta x^{(t)} \begin{cases} x^{(t)}-x^{(t_{d})}&\text{d:even}\\x^{(t)}-x_{0}^{(t_{d})}&\text{d:odd}
\end{cases}$\\
   End If\\
$y^{(t+1)}=u^{(t)} \Delta$, $u^{(t)}==\mathop{\arg\min}_{u\in \mathbb{C}} \ \ \| r^{(t)}-uA\Delta x\|_{2}$\\
6: Update the initial guess $x_{0}^{(t+1)}=x^{(t)}+y^{(t+1)}$\\
7: End For\\
\hline
\end{tabular}
\end{table} \

\section{Techniques for Accelerating the Convergence of Restarted GMRES(m) Based on the Projection}

\subsection{LGMRES(m) method}
Recall the process of the GMRES(m) method: $Az^{(1)}$ is the projection of $r_{0}$ on $Az^{(1)}$, $r_{1}=r_{0}-Az^{(1)}$, $Az^{(2)}$ is the projection of $r_{1}$ on $Az^{(2)}$, $r_{2}=r_{1}-Az^{(2)}$, its idea is shown Figure 1(a). Therefore, the restart GMRES(m) algorithm is equivalent to continuously projecting $r_{i}$ to $Az^{(i+1)}$ ($i=0,1,2,...$).
\begin{figure}[!htbp]
\begin{minipage}[t]{0.5\linewidth}
\centering
\subfigure[The idea of GMRES(m)]{\includegraphics[width=2.7in,height=1.7in]{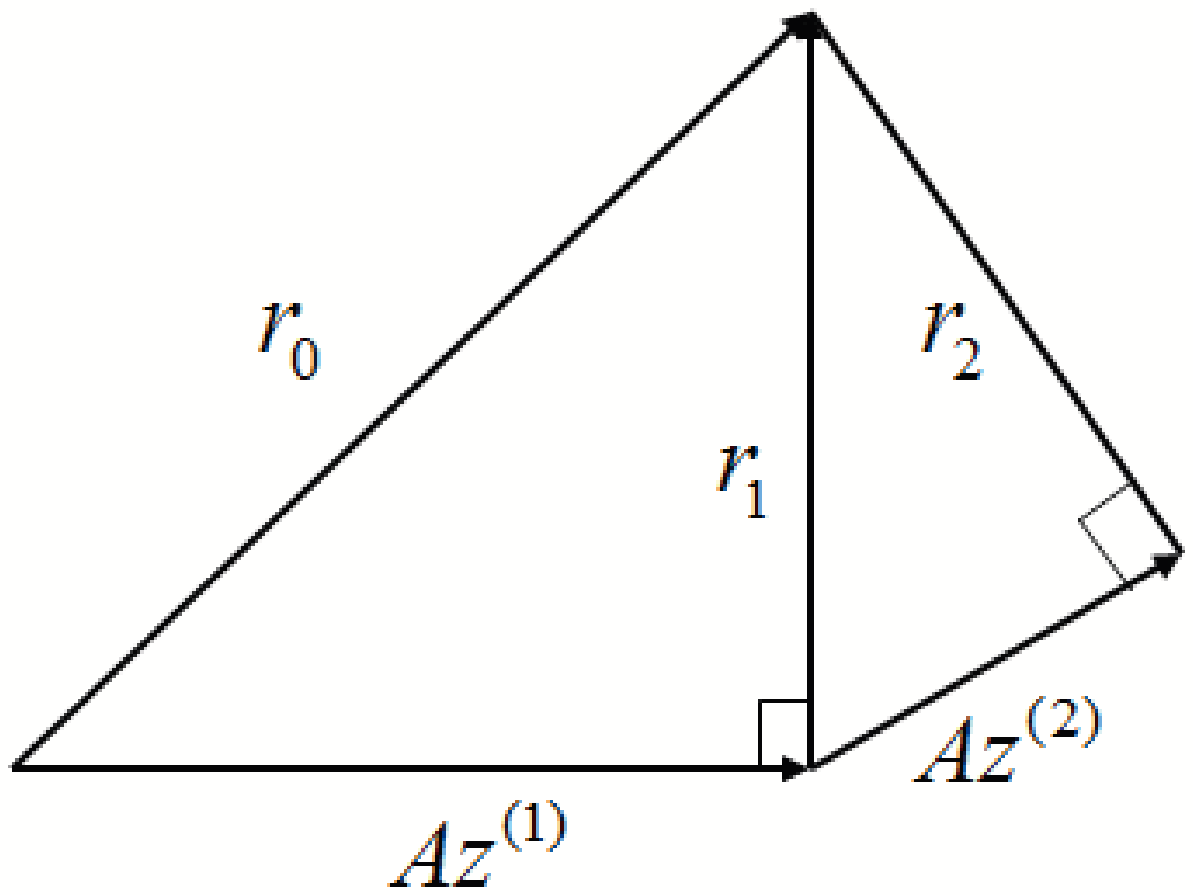}}
\label{fig:side:a}
\end{minipage}%
\begin{minipage}[t]{0.5\linewidth}
\centering
\subfigure[The idea of our method]{\includegraphics[width=2.7in,height=1.7in]{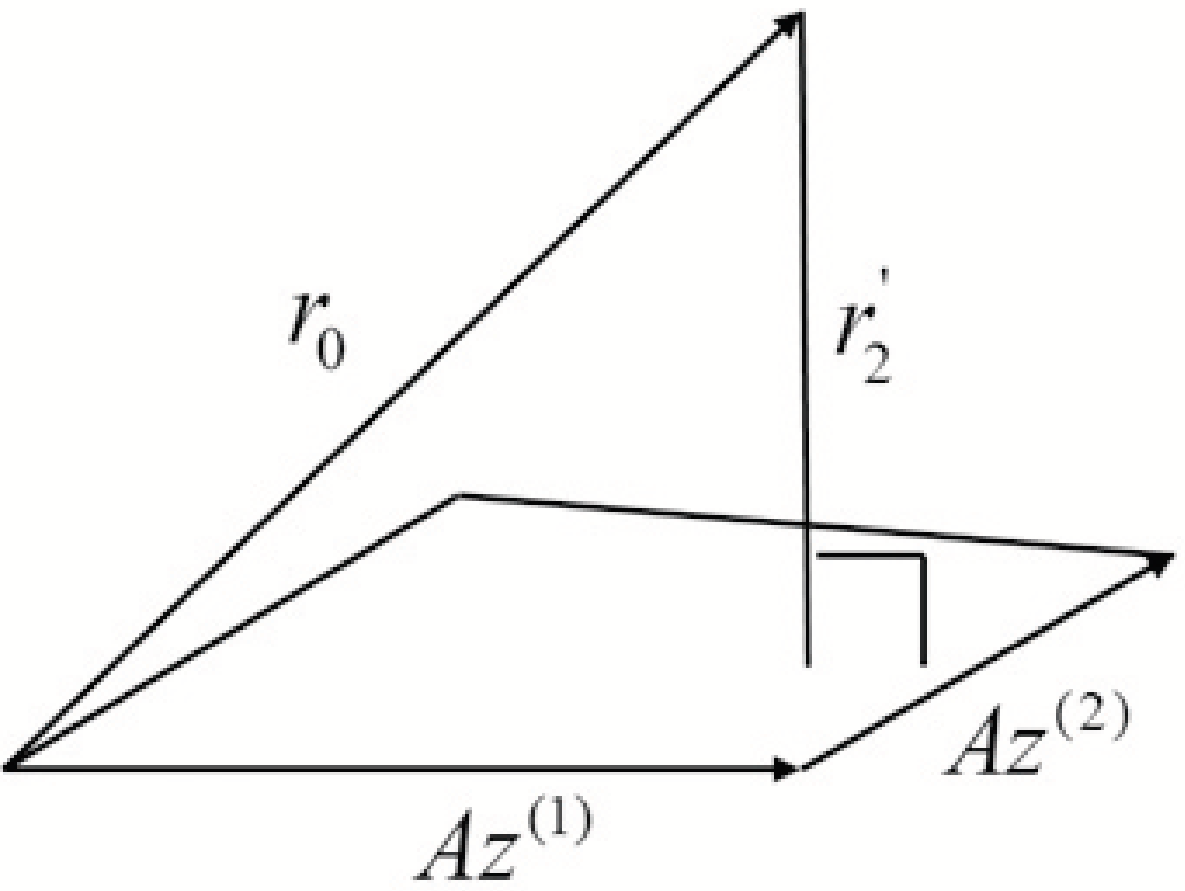}}
\label{fig1}
\end{minipage}
\caption{The relationship of GMRES(m) and our method.}
\end{figure}

If let $r_{2}^{'}$ be the projection of $r_{0}$ on the plane of which $Az^{(1)}$ and $Az^{(2)}$ form, its idea is shown Figure 1(b). So we can easily get $\|r_{2}^{'}\|_{2}\leq \|r_{2}\|_{2}$. According to this idea, let us redesign the restarted GMRES(m) method as follows.

For restarted GMRES(m), $\mathbb{L}_{m}=A\mathbb{K}_{m}$, $W=AV$.
$$x^{(1)}=x_{0}^{(1)}+z^{(1)},$$
$$x^{(2)}=x_{0}^{(2)}+z^{(2)},$$
Induction can be obtained that
$$x^{(l)}=x_{0}^{(1)}+z^{(1)}+z^{(2)}+...+z^{(l)},$$
$$r_{l}=b-Ax^{(l)}=r_{0}-A(z^{(1)}+z^{(2)}+...+z^{(l)}).$$
Set $R=[z^{(1)}\quad z^{(2)}\quad ...\quad z^{(l)}]$, solve
\begin{equation}\label{3.3}
r_{l}^{'}=\mathop{\min}_{y\in R^{l}} \ \ \| b-ARy\|_{2},
\end{equation}
i.e.,
$$R^{T}A^{T}ARy=R^{T}A^{T}b.$$
Suppose that the column of the matrix $R$ is full rank, then
$$y=(R^{T}A^{T}AR)^{-1}R^{T}A^{T}b,$$
so $x{'}^{(t)}=x_{0}^{(1)}+Ry$ and $\|r_{l}^{'}\|_{2}\leq \|r_{l}\|$.
If $y=[1,1,...,1]^{T}_{l\times 1}$, then $r_{l}=ARy$. We take this technique as a restart GMRES(m) base on the projection (abbreviated as LGMRES(m) method, see Algorithm 3.). Since the number $l$ is usually relatively small, the workload of solving equation (\ref{3.3}) is relatively small and fast.




In addition, the relational expression
$$x^{(t)}\in x_{0}^{(1)}+\mathbb{K}_{m\times t}(A,r_{0})$$
is still guaranteed in this iteration process. According to the experiments and the principle of projection, when the linear correlation between $Az^{(1)}$, $Az^{(2)}$, $...$, and $Az^{(l)}$ is weaker or even linearly independent, the acceleration effect is better, that is, the value of $\|r_{l}\|_{2}-\|r_{l}^{'}\|_{2}$ is larger. If the convergent threshold is still not reached after many restarts, which indicates that the linear correlation after each reboot is weak. At this time, select a smaller $l$ at the early stage of the iteration. In the later period of iteration, we can appropriately increase the value of the $l$. This can avoid linear correlation between column vectors of $R$ and speed up its convergence.

\begin{mypro}
 The LGMRES(m) method satisfies the monotonic decrease of the residual 2-norm:
$$\|r_{t+1}^{'}\|_{2}\leq \|r_{t}^{'}\|_{2}, (t=l,l+1,...).$$
\end{mypro}

\begin{proof}
Note that GMRES(m) uses residual vector $r$ to make projection to planar $\mathbb{L}_{m}$, so that the residual norm $\|r\|_{2}$ is monotonically decreasing throughout the iteration process, thus ensuring that the GMRES(m) algorithm is convergent. GMRES method forms residual sequence $r_{t} (t=1,2,...)$ and projection plane sequence $L_{t} (t=1,2,...)$  in the whole iteration process. $Az^{(t)} (t=1,2,..)$ is the projection of $r_{t-1}$ on $L_{t}$. We can easily get:
 $$r_{t}=r_{t-1}-Az^{(t)},\;\; r_{t}\bot Az^{(t)},$$
$\|r_{t}\|_{2}$ is a monotone decreasing sequence.
When $t \geq l$, we take the above acceleration strategy (\ref{3.3}). The residual error is $r_{t} (t=1,2,...)$ before acceleration, and the residual error is $r_{t}^{'}(t=l,l+1,...)$ after acceleration.

According to the principle of acceleration (\ref{3.3}), we can get
$$\|r_{t}^{'}\|_{2}\leq \|r_{t}\|_{2}\;\; (t=l,l+1,...)$$

In addition, according to the principle of GMRES method, we are easy to get
$$\|r_{t+1}\|_{2}\leq \|r_{t}\|_{2}\;\; (t=1,2,...)$$

So $$\|r_{t+1}^{'}\|_{2}\leq \|r_{t}^{'}\|_{2}\;\; (t=l,l+1,...).$$
i.e., $\|r_{t}^{'}\|_{2}\;(t=l,l+1,...)$ is also a monotone decreasing sequence. This ensures the convergence of the algorithm after acceleration, see numerical experiments in Section 5.
\end{proof}

\subsection{Update strategy for $R$}

Let $t$ be the number of iterations and $l$ is the given control parameter. According to (\ref{3.3}), when $t>l$, we replace the first column of $R$ with the newly generated $z^{(l+1)}$, and the newly generated $z^{(l+2)}$ to replace the second column of $R$. Denote $k=mod(t,l)$, $k$ is $t$ divided by the remainder of $l$. Use $z^{(t)}$ to replace the $k$-th column of $R$. If $k=0$, then $z^{(t)}$ takes the place of the last column of $R$, the $R\in \mathbb{R}^{n \times l}$ can always be maintained throughout the iteration. The former $l-1$ step of the LGMRES method is the same as that of the GMRES method. Specific details, see Algorithm 3.

\begin{table}[!htbp]
\begin{tabular}{l}
\hline
Algorithm 3. The LGMRES(m) method (t is an iterative step)\\
\hline
1: Chose the restart frequency $m$ and the initial guess  $x_{0}^{(1)}=0$, $X_{0}=0, b_{0}=b$, $t=0$\\
2: Compute $r_{0}=b-Ax_{0}^{(1)}$\\
3: $t=t+1$\\
4:  If $t<l$\\
5:  Set $\beta=\|r_{0}\|_{2}$\\
6:  Through GMRES algorithm, Obtain $V_{t}s_{t}$, $k=mod(t,l)$\\
7:  If $k=0$, then $k=l$ end if\\
8:  $R_{k}=V_{t}s_{t}$\\
9:  $x_{t}=x_{0}^{t}+V_{t}s_{t}, r_{t}=b-Ax_{t}$\\
10: else $y=(R^{T}A^{T}AR)^{-1}R^{T}A^{T}b$\\
11: $X^{(t)}=X^{(t-1)}+Ry$\\
12: $r_{t}=b-AX^{(t)}$, $x_{0}^{t+1}=x_{0}^{1}$\\
13: end if\\
14: If algorithm is convergence then stop, otherwise go to 3\\
  \hline

\end{tabular}
 \end{table} \

It is also possible to use such a principle of acceleration in Look-Back-type restarted method [6]. According Look-Back-type restarted method, the parameter $d$ is set to $d=3$, which is an odd number, so
$$t_{d}=t-\frac{d-1}{2}=t-1,$$
\begin{eqnarray*} \label{equ20}
\Delta x^{(t)}&=&x^{(t)}-x_{0}^{t_{d}}=x^{(t)}-x_{0}^{(t-1)}=x_{0}^{(t)}+z^{(t)}-x_{0}^{(t-1)}\\
&=&x^{(t-1)}+y^{(t)}+z^{(t)}-x_{0}^{(t-1)}\\
&=&x_{0}^{(t-1)}+z^{(t-1)}+y^{(t)}-x_{0}^{(t-1)}\\
&=&z^{(t-1)}+y^{(t)}+z^{(t)}.
\end{eqnarray*}

Look-Back-type restart technology is combined with LGMRES acceleration technology to get a new method LLBGRES, see Algorithm 4. This method is similar to LGMRES, the different parts are only the part of the GMRES. The part of the change is presented in the following Algorithm 4.

\begin{table}[!htbp]
\begin{tabular}{l}
\hline
Algorithm 4. The LLBGMRES Method\\
\hline
1: 1-6 steps to implement LGMRES algorithm\\
2: Through GMRES algorithm, obtain $V_{t}s_{t}$, $k=mod(t,l)$. The part of the GMRES change:\\
\indent\;\; $x_{0}^{(t+1)}=x^{(t)}\rightarrow x_{0}^{(t+1)}=x^{(t)}+y^{(t+1)}$\\
Computer the vector $y^{(t+1)}$ as follows:\\
\indent\;\;\; If $t=1$ then $y^{(t+1)}=0$\\
\indent\;\;\; else $\Delta x^{(t)}=z^{(t-1)}+y^{(t)}+z^{(t)}$\\
\indent\;\;\; End If\\
\indent\;\;\; $y^{(t+1)}=u^{(t)}\Delta x$, \;$u^{(t)}=\mathop{\arg\min}_{u\in \mathbb{C}} \ \ \| r^{(t)}-uA\Delta x\|_{2}$\\
Update the initial guess $x_{0}^{(t+1)}=x^{(t)}+y^{(t+1)}$\\
3:
7-14 steps to implement LGMRES algorithm\\
  \hline
\end{tabular}
 \end{table} \

\section{Numerical Experiments and Results}

We have proposed an effective acceleration technique of GMRES method and Look-Back-type method by considering the relationship between the changes in each restart of GMRES. To show the potential for efficient convergence, we now present some results from numerical experiments. These experiments are mainly from the linear equations generated by more than 60 practical problems in the matrix market (Available from https://sparse.tamu.edu/).

\subsection{Numerical results}

First, let us observe the convergence curve of the algorithms. The residual 2-norm convergence histories for ACTIVSg2000, COUPLED, KIM1 and MARIO001 are respectively shown in Figure 2 when $m=30$ and Figure 3 when $m=50$. Through these figures, we can know that the LGMRES and LLBGMRES method show that the computation error decreases monotonously with the increase of iteration steps. It also verifies the feasibility of the accelerated scheme from the experimental results. We can see that the acceleration effect of ACTIVSg2000, COUPLED and KIM1 is pretty good and stable. For MARIO001 (m=30), the acceleration of MARIO001 ($m=30$) is not ideal, since the number of iterations of LLBGMRES and LBGMRES is almost the same, but the convergence curve is still monotonously decreasing. From these results, it appears that the LGMRES(m) and LLBGMRES(m) methods may have a high potential for efficient convergence. The acceleration is always good and stable.

\begin{figure}[!hbp]
\begin{minipage}[t]{0.5\linewidth}
\centering
\subfigure[CAVITY10(m=30)]{\includegraphics[width=3.0in,height=2.0in]{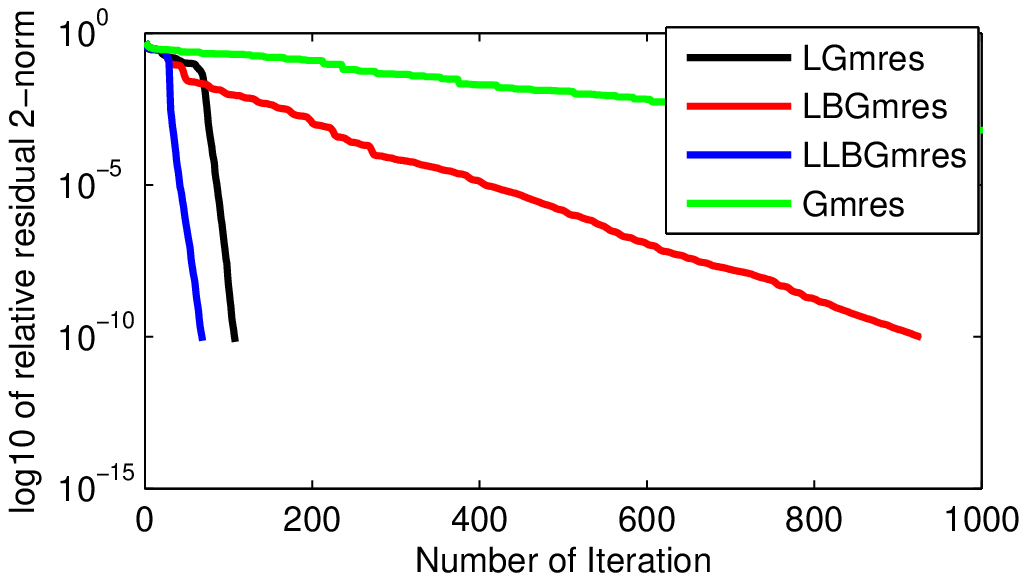}}
\label{fig:side:a}
\end{minipage}%
\begin{minipage}[t]{0.5\linewidth}
\centering
\subfigure[KIM1(m=30)]{\includegraphics[width=3.0in,height=2.0in]{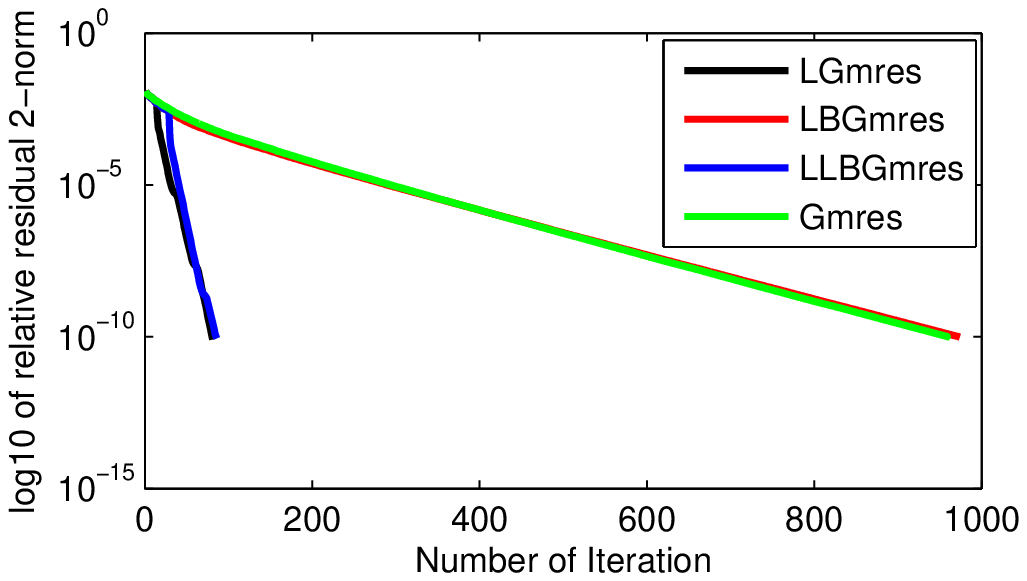}}
\label{}
\end{minipage}
\end{figure}
\begin{figure}[!hbp]
\begin{minipage}[t]{0.5\linewidth}
\centering
\subfigure[ACTIVSg2000(m=30))]{\includegraphics[width=3.0in,height=2.0in]{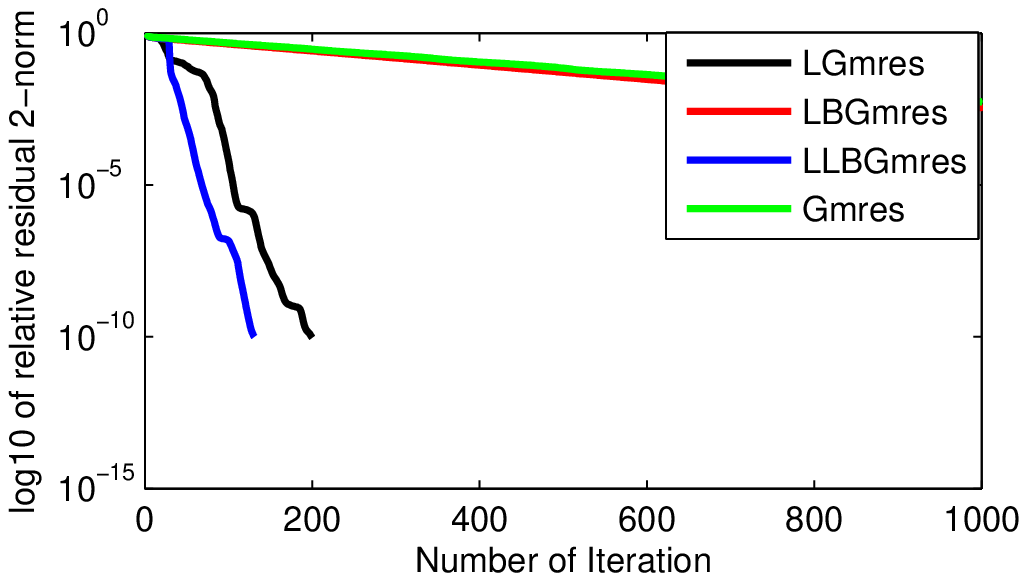}}
\label{fig:side:a}
\end{minipage}%
\begin{minipage}[t]{0.5\linewidth}
\centering
\subfigure[MARIO001(m=30)]{\includegraphics[width=3.0in,height=2.0in]{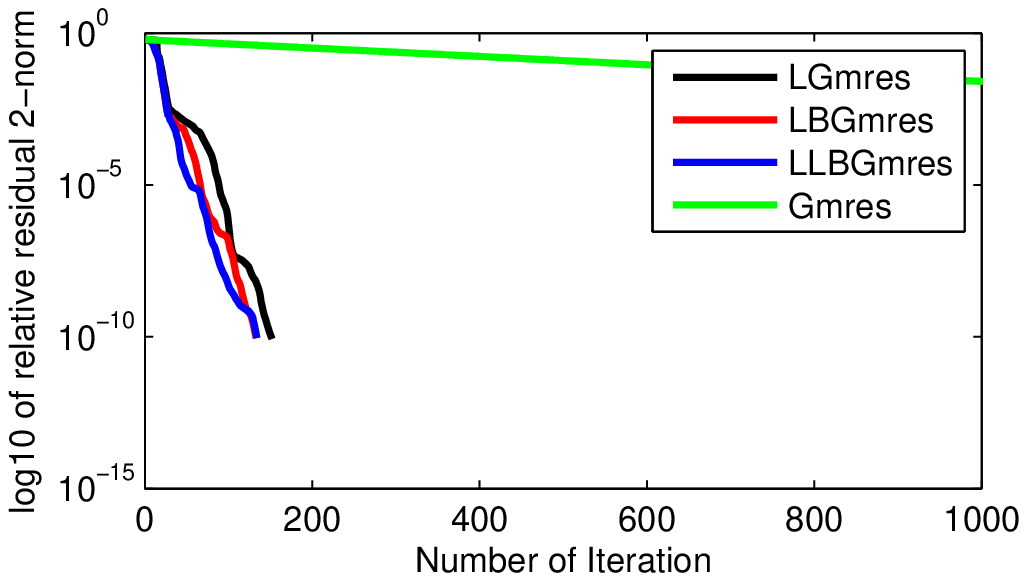}}
\label{}
\end{minipage}
\caption{The relative residual 2-norm history for ACTIVSg2000 ,COUPLED, KIM1 and MARIO001,when m=30}
\end{figure}

\begin{figure}[!hbp]
\begin{minipage}[t]{0.5\linewidth}
\centering
\subfigure[CAVITY10(m=50)]{\includegraphics[width=3.0in,height=2.0in]{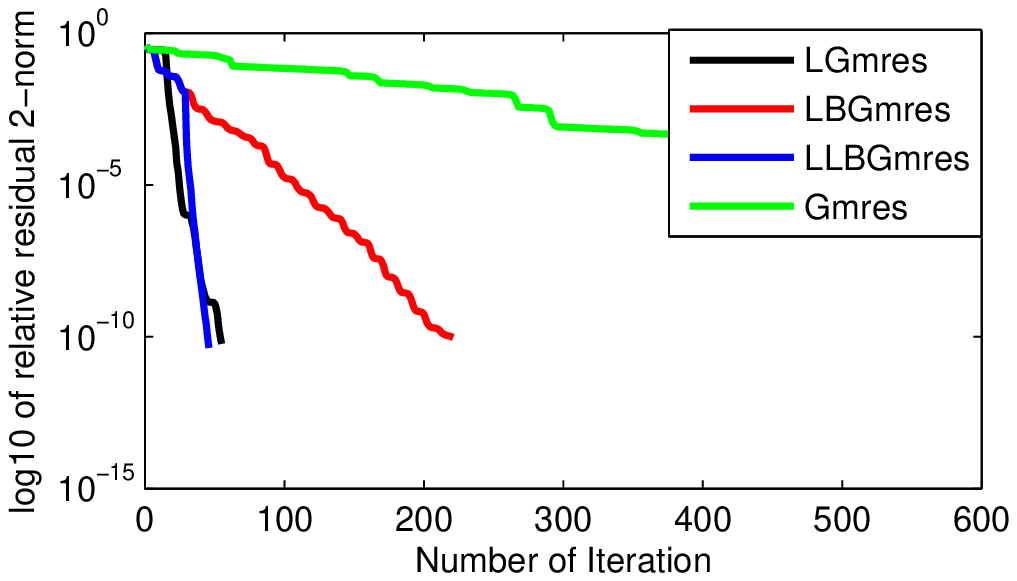}}
\label{fig:side:a}
\end{minipage}%
\begin{minipage}[t]{0.5\linewidth}
\centering
\subfigure[KIM1(m=50)]{\includegraphics[width=3.0in,height=2.0in]{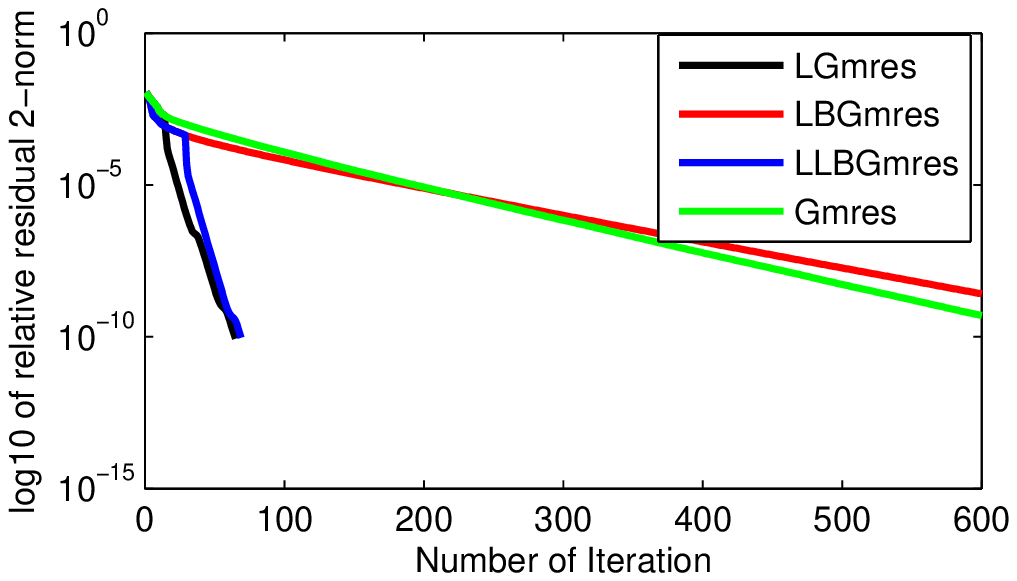}}
\label{}
\end{minipage}
\end{figure}

\begin{figure}[!htbp]
\begin{minipage}[t]{0.5\linewidth}
\centering
\subfigure[ACTIVSg2000(m=50)]{\includegraphics[width=3.0in,height=2.0in]{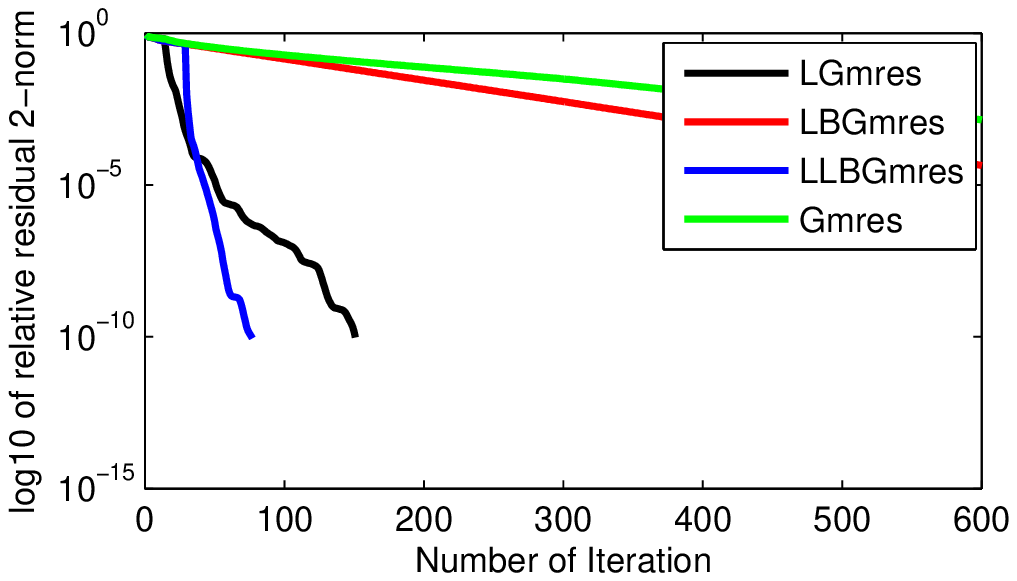}}
\label{fig:side:a}
\end{minipage}%
\begin{minipage}[t]{0.5\linewidth}
\centering
\subfigure[MARIO001(m=50)]{\includegraphics[width=3.0in,height=2.0in]{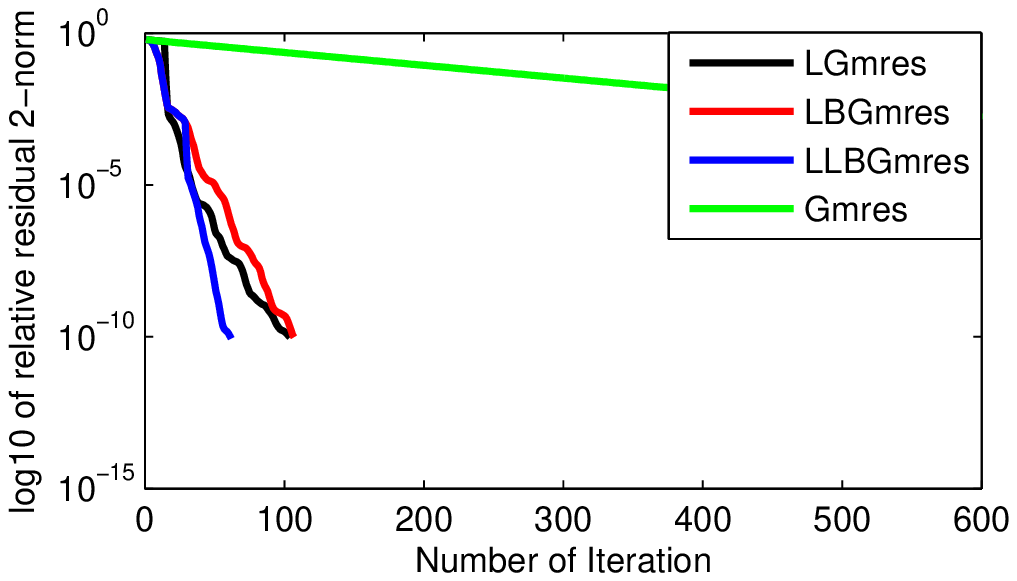}}
\label{}
\end{minipage}
\caption{The relative residual 2-norm history for ACTIVSg2000,COUPLED, KIM1 and MARIO001,when m=50}
\end{figure}

Next, we analyze the numerical results in terms of three aspects: the number of iterations, computation time per one restart cycle ($m$ iterations) and total computation time. The results for $m=30$ and $m=50$ are presented in Tables 1 and 2 respectively for different matrix classes.

First, we analyze the number of iterations(Iter) from the four methods for the different variant of the GMRES(m) method. In most cases, the LGMRES(m) and LLBGMRES(m) methods shows almost the same or lower Iter than the GMRES(m) and LBGMRES(m) methods. In particular, for ACTIVSg2000, COUPLED($m=30$), LGMRES(m) and LLBGMRES methods converged after about 100 restarts, but GMRES(m) and LBGMRES(m) did not converge after 1000 restarts. For COUPLED($m=30$), CAVITY10, RAJAT03, KIM1, Goodwin-010, the number of restart cycles for LGMRES(m) and LLBGMRES methods is much less than that of GMRES(m) and LBGMRES(m). For CHIPCOOL0($m=30$) and Goodwin-017($m=30$), these four methods have not reached the conditions of convergence. But LGMRES(m), LBGMRES(m) and LLBGMRES methods converge for CHIPCOOL0($m=50$) and Goodwin-017($m=50$). The number of LGMRES(m) and LLBGMRES methods restarts is far less than LBGMRES method. For MARIO001($m=30$), the number of LGMRES(m) method restarts is more than LBGMRES method, but the number of LLBGMRES(m) method restarts is less than LBGMRES method for MARIO001($m=50$). We can see that the smaller restart frequency $m$ leads to a larger difference in Iter between these methods.

Next, we consider the computation time per restart cycle(t-Restar). For $m=30$, LLBGMRES method is the most one, LGMRES method is second, and LBGMRES method is third. For $m=50$, the gap between their consumption of time is reduced. In terms of the total computation time(t-Total), the time used by LGMRES and LLGMRES is relatively small. Although they have more time per cycle, they can quickly converge. So the total time used is relatively small. For MARIO001($m=30$), the speed of acceleration is not ideal because LLBGMRES takes much more time per cycle, resulting in more total consumption time.


 \begin{table}[!htpb]
 \caption{\scriptsize Test problems (n: order of matrix, Nnz : number of nonzeros in matrix) and  convergence results (Iter: number of restarts, tTotal: total omputation time, tRestart:
computation time per one restart cyle) of the GMRES(m), LBGMRES, LGMRES and LLBGMRES, where $m=30$.}
\label{Tab:01}
 \centering
 \begin{tabular}{c c c c c c } 
 \hline
 Matrix & & Solver & Iter & Time[sec.]& \\
 n & Nnz & & & t-Total & t-Restar\\
 \hline
 COUPLED & & LGMRES & 104 & 8.58e0 & 8.25e-2\\
11341 & 21199 & LBGMRES & \dag & 1.29e2 & 1.29e-1\\
 &  &  LLBGMRES & 92 & 8.41e0 & 9.14e-2\\
 &  & 	GMRES & \dag & 7.64e1 & 7.64e-2\\
 \hline
CAVITY10 &  &  LGMRES & 108 & 4.76e0 & 4.40e-2\\
2597 & 76367 &  LBGMRES & 927 & 2.56e1 & 2.76e-2\\
 &   & LLBGMRES & 69 & 2.07e0 & 3.00e-2\\
 &  & GMRES &\dag & 2.92e1 & 2.92e-2\\
 \hline
RAJAT03 &  &  LGMRES & 79 & 6.18e0 & 7.82e-2\\
7602 & 32653 & LBGMRES & 439 & 2.47e1 & 5.62e-2\\
 &  & LLBGMRES & 41 & 2.41e0 & 5.87e-2\\
 &  & GMRES & \dag & 4.87e1 & 4.87e-2\\
 \hline
KIM1 &  &  LGMRES & 81 & 6.08e1 & 7.50e-1\\
38415 & 933195 & LBGMRES & 974 & 6.40e2 & 6.57e-1\\
 &  & LLBGMRES & 86 & 6.93e1 & 8.05e-1\\
 &  & GMRES & 961 & 6.04e2 & 6.28e-1\\
 \hline
NS3DA &  &  LGMRES & 68 & 2. 63e1 & 3.86e-1\\
20414 & 1679599 & LBGMRES & 78 & 2.49e1 & 3.19e-1\\
 &  & LLBGMRES & 60 &  2.41e1 & 4.01e-1\\
 &  & GMRES & 78 &  2.46e1 &  3.15e-1\\
 \hline
CHIPCOOL0 &  &  LGMRES & \dag & 1.85e2 & 1.85e-1\\
20082 & 281150 & LBGMRES & \dag & 1.28e2 & 1.28e-1\\
 &  & LLBGMRES & \dag & 2.08e2 & 2.08e-1\\
 &  & GMRES & \dag & 1.33e2 & 1.33e-1\\
 \hline
WAVEGUIDE3D &  & LGMRES & 928 & 5.56e2 & 5.99e-1\\
21036 & 303468 & LBGMRES & \dag & 4.09e2 & 4.09e-1\\
 &  & LLBGMRES & 912 & 5.53e2 & 6.06e-1\\
 &  & GMRES & \dag & 3.88e2 & 3.88e-1\\
 \hline
MEMPLUS &  &  LGMRES &  76 & 6.99e0 & 9.19e-2\\
17758 & 99147 &  LBGMRES &  76 & 6.57e0 & 8.64e-2\\
 &  & LLBGMRES &  73 & 7.55e0 & 1.03e-1\\
 &  & GMRES &  221 & 2.07e1 & 9.36e-2\\
 \hline
MAJORBASIS &  &  LGMRES & 7 &  8.81e0 &  1.25e0\\
160000 & 1750416 &  LBGMRES & 7 & 9.35e0 &  1.33e0\\
 &  & LLBGMRES & 7 & 9.30e0 & 1.32e0\\
 &  & GMRES &  7 & 9.20e0 &  1.31e0\\
 \hline
PFINAN512 &  &   LGMRES &  28 &  1.17e1 & 4.17e-1\\
74752 & 596992 & LBGMRES & 121 & 6.44e1 & 5.32e-1\\
 &  & LLBGMRES & 34 & 1.50e1 &  4.41e-1\\
 &  & GMRES & \dag &  4.42e2 &  4.42e-1\\
 \hline
\end{tabular}
 \end{table} \
\begin{table}[!htbp]
 \centering
 \begin{tabular}{c c c c c c } 
  \hline
WANG3 &  & LGMRES & 21 & 2.96e0 &  1.40e-1\\
26064 & 177168 & LBGMRES &  26 & 3.96e0 & 1.52e-1\\
 &  & LLBGMRES & 26 &  4.96e0 & 1.90e-1\\
 &  & GMRES & 28 &  3.92e0 & 1.40e-1\\
 \hline
THERMAL1 &  &  LGMRES & 98 &  5.44e1 &  5.55e-1\\
82654 & 574458 & LBGMRES & 121 &  6.17e
1 &  5.09e-1\\
 &  & LLBGMRES & 94 &  5.77e1 &  6.13e-1\\
 &  & GMRES & \dag & 4.71e2 & 4.71e-1\\
 \hline
EPB1 &   & LGMRES & 44 &  3.73e0 & 8.47e-2\\
14734 & 95053 &  LBGMRES &  67 &  5.63e0 & 8.40e-2\\
 &  & LLBGMRES &  38 & 3.22e0 & 8.47e-2\\
 &  & GMRES &  93 &  7.92e0 &  8.51e-2\\
 \hline
MARIO001 &  &  LGMRES & 152 &  4.14e1 &  2.72e-1\\
38434 & 204912 &  LBGMRES & 133 & 2.98e1 &  2.24e-1\\
 &  & LLBGMRES &  134 & 4.16e1 & 3.10e-1\\
 &  & GMRES & \dag & 2.21e2	 & 2.21e-1\\
 \hline
TUMA2 &  &  LGMRES & 211 & 1.84e1 &  8.72e-2\\
12992 & 49365 &  LBGMRES &  357 & 2.53e1 &  7.08e-2\\
 &  & LLBGMRES & 198 & 1.89e1 & 9.54e-2\\
 &  & GMRES & \dag & 7.92e1 & 7.93e-2\\
 \hline
ACTIVSg2000 &  &   LGMRES & 200 & 6.02e0 & 3.01e-2\\
4000 & 28505 & LBGMRES & \dag & 2.41e1 & 2.41e-2\\
 &  & LLBGMRES & 131 &  3.76e0 & 2.87e-2\\
 &  & GMRES &  \dag& 2.44e1 & 2.44e-2\\
 \hline
Goodwin 010 &  & LGMRES & 26 &  2.70e-1 & 1.03e-2\\
1182 & 32282 &  LBGMRES & 194 & 1.64e0 &  8.36e-3 \\
 &  & LLBGMRES & 36 & 3.45e-1 & 9.83e-3\\
 &  & GMRES &  454 &  3.86e0 &  8.50e-3\\
 \hline
Goodwin 017 &  &  LGMRES & \dag & 4.18e1 &  4.18e-2\\
3,317 & 97,773 &  LBGMRES & \dag & 4.30e1 &  4.30e-2\\
 &  & LLBGMRES & \dag & 3.40e1 &  3.40e-2\\
 &  & GMRES & \dag & 3.22e1 & 3.22e-2\\
\hline
 \end{tabular}
 \end{table} \
\begin{table}[!htpb]
\caption{\scriptsize Test problems(n: order of matrix, Nnz : number of nonzeros in matrix) and convergence results (Iter: number of restarts, tTotal: total omputation time, tRestart:  computation time
per one restart cyle) of the GMRES(m), LBGMRES, LGMRES and LLBGMRES, where $m=50$.}
\label{Tab:02}
 \centering
 \begin{tabular}{c c c c c c } 
  \hline
  Matrix &  & Solver & Iter & Time[sec.] &\\
n & Nnz &  &  & t-Total & t-Restar\\
  \hline
COUPLED &  & LGMRES & 57 & 1.16e1 & 2.03e-1\\
11341 & 21199 & LBGMRES & 528 & 9.90e1 & 1.87e-1\\
 &  & LLBGMRES & 69 & 1.44e1 & 2.08e-1\\
 &  & GMRES & \dag & 1.05e2 & 1.75e-1\\
   \hline
CAVITY10 &  &  LGMRES &  55 & 3.47e0 & 6.3e-2\\
2597 & 76367 & LBGMRES & 221 & 1.23e1 & 5.56e-2\\
 &  & LLBGMRES & 46 & 2.77e0 & 6.02e-2\\
 &  & GMRES & \dag & 6.62e1 & 1.10e-1\\
   \hline
RAJAT03 &  &  LGMRES &  73 &  7.56e0 &  1.03e-1\\
7602 & 32653 &  LBGMRES & 187 & 2.14e1 & 1.14e-1\\
 &  & LLBGMRES & 35 & 3.77e0 & 1.07e-1\\
 &  & GMRES & \dag & 6.48e1 & 1.08e-1\\
   \hline
KIM1 &  &  LGMRES &  65 & 9.44e1 &  1.45e1\\
38415 & 933195 &  LBGMRES & \dag &  8.61e2 & 1.43e1\\
 &  & LLBGMRES & 69 & 1.06e2 & 1.53e1\\
 &  & GMRES & \dag & 8.27e2 & 1.37e1\\
   \hline
NS3DA &  &  LGMRES &  37 &  2.18e1 &  5.89e-1\\
20414 & 1679599 &  LBGMRES &  39 &  2.23e1 &  5.71e-1\\
 &  & LLBGMRES & 33 & 2.00e1 & 6.00e-1\\
 &  & GMRES & 47 & 2.89e1 & 6.14e-1\\
   \hline
CHIPCOOL0 &  & LGMRES & 513 &  1.48e2 & 2.88 e-1\\
20082 & 281150 & LBGMRES & 555 &  1.75e2 & 3.15 e-1\\
 &  & LLBGMRES & 479 & 1.55e2 & 3.23 e-1\\
 &  & GMRES & \dag & 1.24e2 & 2.08e-1\\
   \hline
WAVEGUIDE3D &  &  LGMRES &  589 & 6.18e2 &  1.04e0\\
21036 & 303468 &  LBGMRES & \dag &  5.36e2 &  8.93e-1\\
 &  & LLBGMRES & 591 & 6.27e2 & 1.06e0\\
 &  & GMRES & \dag & 5.39 e2 & 8.98e-1\\
   \hline
MEMPLUS	 &  &  LGMRES &  43 & 1.11e1 & 2.50e-1\\
17758 & 99147 & LBGMRES & 39 &  8.14e0 & 2.08e-1\\
 &  & LLBGMRES & 38 & 8.33e0 & 2.19e-1\\
 &  & GMRES & 118 & 2.89e1 & 2.44e-1\\
   \hline
MAJORBASIS &  &  LGMRES & 5 &  2.25e1 &  4.50e-1\\
160000 & 1750416 &  LBGMRES & 5 & 2.23e1 &  4.46e-1\\
 &  & LLBGMRES & 5 & 1.94e1 & 3.88e-1\\
 &  & GMRES & 5 & 1.88e1 & 3.76e-1\\
   \hline
PFINAN512 &  &  LGMRES &  23 &  3.16e1 &  1.37e0\\
74752 & 596992 &  LBGMRES &  58 & 7.72e1 &  1.33e0\\
 &  & LLBGMRES & 32 & 3.99e1 & 1.24e0\\
 &  & GMRES & \dag & 7.65e2 & 1.27e0\\
\hline
\end{tabular}
\end{table} \

\newpage
\begin{table}[!htpb]
\centering
\begin{tabular}{c c c c c c } 
\hline
WANG3 &  &  LGMRES &  16 &  5.78e0 &  3.61e-1\\
26064 & 177168 &  LBGMRES &  15 & 4.70e0 &  3.13e-1\\
 &  & LLBGMRES & 15 & 5.89e0 & 3.92e-1\\
 &  & GMRES & 19 & 7.50e0 & 3.94e-1\\
   \hline
THERMAL1 &  &  LGMRES &  51 &  7.08e1 &  1.38e0\\
82654 & 574458 & LBGMRES &  62 &  9.00e1 &  1.45e0\\
 &  & LLBGMRES & 43 & 6.75e1 & 1.56e0\\
 &  & GMRES &  \dag& 7.01e2 & 1.16e0\\
   \hline
EPB1 &  &  LGMRES & 25 & 5.41e0	 & 2.16e-1\\
14734 & 95053 &  LBGMRES &  36 &  6.86e0 &  1.90e-1\\
 & &  LLBGMRES & 31 & 5.83e0 & 1.88e-1\\
 &  & GMRES & 41 & 8.74e0 & 2.13e-3\\
   \hline
MARIO001 &  &  LGMRES & 104 &  6.21e1 &  5.97e-1\\
38434 & 204912 &  LBGMRES &  107 &  7.38e1 &  6.89e-1\\
 &  & LLBGMRES & 62 & 3.45e1 & 5.56e-1\\
 &  & GMRES &\dag  & 3.67e2 & 6.11e-1\\
   \hline
TUMA2 &  &  LGMRES &  186 &  4.32e1 &  2.32e-1\\
12992 & 49365 & LBGMRES &  232 & 5.44e1 & 2.34e-1\\
 &  & LLBGMRES & 129 & 2.88e1 & 2.23e-1\\
 &  & GMRES & \dag & 1.29e2 & 2.15e-1\\
   \hline
ACTIVSg2000 &  &  LGMRES &  151 &  1.61e1 &  1.06e-1\\
4000 & 28505 &  LBGMRES	 &\dag & 4.78e1 &  7.96e-2\\
 &  & LLBGMRES & 77 & 5.75e0 & 7.23e-2\\
 &  & GMRES & \dag & 4.97e1 & 8.28e-2\\
   \hline
Goodwin 010 &  &  LGMRES &  22 &  5.02e-1 & 2.28e-2\\
1182 & 32282 &  LBGMRES & 83 & 2.66e0 & 3.20e-2\\
 &  & LLBGMRES & 33 & 9.86e-1 & 2.98e-2\\
 &  & GMRES &  165 &  2.65e0 & 1.60e-2\\
   \hline
Goodwin 017 &  &  LGMRES &  27 &  1.93e0 &  7.14e-2\\
3,317 & 97,773 &  LBGMRES & 496 & 3.83e1 & 7.72e-2\\
 &  & LLBGMRES & 38 & 3.17e0 & 8.34e-2\\
 &  & GMRES & \dag &  4.96e1 &  7.81e-2\\

  \hline
 \end{tabular}
 \end{table} \

\section{Conclusion}

In this paper, GMRES(m) and Look-Back GMRES(m) method are speeded up with the idea of projection, and the acceleration effect is quite stable and satisfactory. Consider the relationship between the increments of  $x$ at each restart. Reassemble $z^{(1)}$, $z^{(2)}$,$\ldots$, and $z^{(l)}$ to achieve their optimal combination, so that the residual norm decreases again. Numerical experiments showed it has the more efficient convergence than the GMRES(m) method widely used to solve large sparse nonsymmetric linear systems. For some problems, the value of their $l$ may be appropriately larger so that the speed of convergence will be faster. In the future work, how to establish an adaptive $l$ value that varies with different problems will be investigated. We will theoretically analyze the acceleration principle and make it better in performance.


\begin{thebibliography}{32}
\bibitem{1}Shufang Xu, Gao Li and Pingwen Zhang. Numerical Linear Algebra. Peking University Press. 2000. pp:139-160.
\bibitem{2} Simoncini V, Szyld D B. Recent computational developments in Krylov subspace methods for linear systems. Numerical Linear Algebra with Applications, 2010, 14(1):1-59.
\bibitem{3} Saad Y. Krylov subspace methods for solving large unsymmetric linear systems. Mathematics of Computation, 1981, 37(155):105-126.
\bibitem{4} Saad Y, Schultz M H. GMRES: a generalized minimal residual algorithm for solving nonsymmetric linear systems. Siam Journal on Scientific and Statistical Computing, 2006, 7(3):856-869.
\bibitem{5}Bai Z Z. Motivations and realizations of Krylov subspace methods for large sparse linear systems. Journal of Computational and Applied Mathematics, 2015, 283:71-78.
\bibitem{6}A Imakura,Tomohiro Sogabe and Shao-Liang Zhang. A look-back-type restart for the restarted krylov subspace methods to solve non-hermitiao systems. Japan Journal of Industrial and Applied Mathematics, 2018, 35(2): 835-859.
\bibitem{7}Imakura A, Tomohiro S, Zhang S L. An Efficient Variant of the GMRES(m) Method Based on the Error Equations. East Asian Journal on Applied Mathematics, 2012, 2(1):19-32.
\bibitem{8}Benzi M, Bertaccini D. Approximate Inverse Preconditioning for Shifted Linear Systems. BIT, 2003, 43(2):231-244.

\bibitem{9} Benzi M, Meyer C D, Tuma M. A Sparse Approximate Inverse Preconditioner For the Conjugate Gradient Method. Siam Journal on Scientific Computing, 1996, 17(5):1135-1149.
\bibitem{10} Benzi M, Cullum J K, Tuma M. Robust Approximate Inverse Preconditioning for the Conjugate Gradient Method. Siam J. sci. comput, 2000, 22(4):1318-1332.
\bibitem{11}Benzi M, Tuma M. A sparse approximate inverse preconditioner for nonsymmetric linear systems. Siam J. comput, 2006, 19(3):1135-1149.
\bibitem{12} Baglama J, Calvetti D, Golub G H, et al. Adaptively Preconditioned GMRES Algorithms. Siam Journal on Scientific Computing, 1998, 20(1):243-269.
\bibitem{13}Bakhos T, Kitanidis P, Ladenheim S, et al. Multipreconditioned GMRES for Shifted Systems. Siam Journal on Scientific Computing, 2017, 39(5).
\bibitem{14}Imakura A, Sogabe T, Zhang S. An Efficient Variant of the Restarted Shifted GMRES Method for Solving Shifted Linear Systems. Journal of Mathematical Research with Applications, 2013, 33(2):127-141.
\bibitem{15}A-Lun G U, Sun Y G. Multidimensional Project Algorithm for Solving Linear Equations. Operations Research And Management Science, 2006, 83(86):79¨C102.
\bibitem{16}Gang W U. A semi-refined biorthogonalization Lanczos method. Journal of Dalian University of Technology, 2002.
\bibitem{17} Eisenstat S C, Elman H C, Schultz M H. Variational Iterative Methods for Nonsymmetric Systems of Linear Equations. Siam Journal on Numerical Analysis, 1983, 20(2):345-357.
\bibitem{18} Liu Q, Morgan R B, Wilcox W. Polynomial Preconditioned GMRES and GMRES-DR. Siam Journal on Scientific Computing, 2015, 37(5): 407-S428.
\bibitem{19} Golub G H, Ye Q. Inexact Inverse Iteration for Generalized Eigenvalue Problems. Bit Numerical Mathematics, 2000, 40(4):671-684.
\bibitem{20} Baglama J, Reichel L. Augmented GMRES-type methods. Numerical Linear Algebra with Applications, 2010, 14(4):337-350.
\bibitem{21} Baker A H, Jessup E R, Kolev T V. A simple strategy for varying the restart parameter in GMRES(m). Elsevier Science Publishers B. V. 2009.
\bibitem{22}Baker A H, Jessup E R, Manteuffel T. A Technique for Accelerating the Convergence of Restarted GMRES. Society for Industrial and Applied Mathematics, 2005.
\bibitem{23}Bayliss A, Goldstein C I, Turkel E. Iterative method for the Helmholtz equation. Journal of Computational Physics, 1983, 49(3):443-457.
\bibitem{24} Burrage K, Erhel J. On the performance of various adaptive preconditioned GMRES strategies. Numerical Linear Algebra with Applications, 2015, 5(2):101-121.
\bibitem{25}Silvester D, Elman H, Kay D, et al. Efficient preconditioning of the linearized Navier-Stokes equations for incompressible flow. Journal of Computational and Applied Mathematics, 2001, 128(1):261-279.
\bibitem{26} Simoncini V, Gallopoulos E. Convergence properties of block GMRES and matrix polynomials. Linear Algebra And Its Applications, 1994, 247(6):97-119.
\bibitem{27} Simoncini V, Szyld D B. On the Occurrence of Superlinear Convergence of Exact and Inexact Krylov Subspace Methods. SIAM Review, 2005, 47(2):247-272.
\bibitem{28}Simoncini V, Szyld D B. Recent computational developments in Krylov subspace methods for linear systems. Numerical Linear Algebra with Applications, 2010, 14(1):1-59.
\bibitem{29} Stoll M, Wathen A. Combination Preconditioning and the Bramble-Pasciak Preconditioner. Society for Industrial and Applied Mathematics, 2008.
\bibitem{30}Trltzsch F. Optimal Control of Partial Differential Equations: Theory, Methods and Applications. Siam Journal on Control and Optimization, 2010, 112(2):399.
\bibitem{31}Y Saad. Iterative Methods for Sparse Linear Systems, SIAM, Philadelphia, PA, second ed., 2003.
\bibitem{32}MATRIX MARKET, Available from https://sparse.tamu.edu/
\end{thebibliography}
\end{document}